\documentclass[12pt]{amsart}

\usepackage{graphicx,eepic, bm}%, times}
\usepackage{amsmath,amsfonts,amssymb}
\newcommand{\be}{\beta}
\newcommand{\de}{\delta}

\newcommand{\al}{\alpha}
\newcommand{\ga}{\gamma}
\newcommand{\e}{\varepsilon}

\newcommand{\Si}{\Sigma}

\newcommand{\BN}{\mathbb{N}}

\newcommand{\ov}{\overline}

\renewcommand{\phi}{\varphi}

\newtheorem{lemma}{Lemma}[section]

\newtheorem{prop}[lemma]{Proposition}
\newtheorem{thm}[lemma]{Theorem}
\newtheorem{cor}[lemma]{Corollary}
\theoremstyle{definition}
\newtheorem{Def}[lemma]{Definition}
\newtheorem{example}[lemma]{Example}

\theoremstyle{remark}
\newtheorem{rmk}[lemma]{Remark}

\numberwithin{equation}{section} \numberwithin{table}{section}

\title[Supercritical holes for the doubling map]
{Supercritical holes for the doubling map}
\author{Nikita Sidorov}
\address{
School of Mathematics, The University of Manchester,
Oxford Road, Manchester M13 9PL, United Kingdom. E-mail:
sidorov@manchester.ac.uk}

\date{\today}

\begin{document}

\begin{abstract}
For a map $S:X\to X$ and an open connected set ($=$ a hole) $H\subset X$ we define $\mathcal J_H(S)$ to be the set of points in $X$ whose $S$-orbit avoids $H$. We say that a hole $H_0$ is supercritical if
\begin{enumerate}
\item for any hole $H$ such that $\ov{H_0}\subset H$ the set $\mathcal J_H(S)$ is either empty or contains only fixed points of $S$;
\item for any hole $H$ such that $\ov H\subset H_0$ the Hausdorff dimension of $\mathcal J_H(S)$ is positive.
\end{enumerate}
The purpose of this note is to completely characterize all supercritical holes for the doubling map $Tx=2x\bmod1$.
\end{abstract}

\maketitle

\section{Introduction and initial results}

The present paper is concerned with an area of dynamics which is usually referred to as ``open dynamical systems'' or, more colloquially, ``maps with holes''. Let us remind the reader the basic set-up: let $X$ be a compact (or precompact) metric space and $S:X\to X$ be a map with positive topological entropy (for the definition of topological entropy see, e.g., \cite[Definition~3.1.3]{KH}). Let $H\subset X$ be an open connected set which we treat as a {\em hole}.

We denote by $\mathcal J_H(S)$ the set of all points in $X$ whose $S$-orbit does not intersect $H$. In other words,
\[
\mathcal J_H(S)=X\setminus \bigcup_{n=-\infty}^\infty S^{-n} H
\]
if $S$ is invertible or
\[
\mathcal J_H(S)=X\setminus \bigcup_{n=0}^\infty S^{-n} H
\]
if it is not. Clearly, $\mathcal J_H(S)$ is $S$-invariant, and in a number of recent papers certain dynamical properties of the  exclusion map $S|_{\mathcal J_H(S)}$ have been studied -- see, e.g., \cite{BKT} and the references therein.

We believe a more immediate issue here is the ``size'' of the set $\mathcal J_H(S)$ -- after all, if it is countable (or, even worse, empty), any questions regarding the dynamics of the exclusion map become uninteresting.

Let $X=[0,1)$ and $Tx=2x\bmod1$, the famous doubling map. Assume first that $H$ is a symmetrical interval about $1/2$, i.e., $H=H_a:=(a,1-a)$ for some $a\in(0,1/2)$. Note first that if $a<1/3$, then $\mathcal J_{H_a}(T)=\{0\}$, since for any $x$ in $(a/2^n,a/2^{n-1})$ or in $(1-a/2^{n-1}, 1-a/2^n)$ with $n\ge1$, we have $T^n(x)\in (a,2a)$ or $(1-2a,1-a)$, both being subsets of $H_a$.

On the other hand, if $a\ge1/3$, then the 2-cycle $\{1/3, 2/3\}$ does not intersect $H_a$. Let us give a general definition:

\begin{Def}\label{def:first}
Let $\mbox{Fix}(S)=\{x\in S : Sx=x\}$. We say that a hole $H_0$ is {\em first order critical} for $S$ if
\begin{enumerate}
\item for any hole $H$ such that $\ov{H_0}\subset H$ we have $\mathcal J_H(S)\subset \mbox{Fix}(S)$;
\item for any hole $H$ such that $\ov H\subset H_0$ we have $\mathcal J_H(S)\not\subset \mbox{Fix}(S)$.
\end{enumerate}
\end{Def}

Thus, $H_{1/3}$ is a first order critical hole for $T$. Note that if $a\in(1/3, 2/5)$, then $\mathcal J_{H_a}(T)=\mathcal J_{H_{1/3}}(T)$, i.e., just the 2-cycle and all its preimages -- see \cite{ACS}. If we keep increasing $a$ towards $1/2$, we start getting more and more cycles in $\mathcal J_{H_a}(T)$. 

  Let us introduce some symbolic dynamics. Namely, put $\Sigma=\{0,1\}^\BN$ and let $\sigma$ denote the one-sided shift on $\Sigma$, i.e., $\sigma(w_1,w_2,w_3,\dots):=(w_2,w_3,\dots)$. Let $\pi:\Sigma\to[0,1]$ be defined as follows:
\[
\pi(w_1,w_2,\ldots)=\sum_{k=1}^\infty w_k2^{-k}.
\]
Then, as is well known, $\pi$ in one-to-one on all sequences except a countable set, and $\pi\sigma=T\pi$. Thus, there is one-to-one correspondence between the cycles for $T$ and those for $\sigma$. If $a\ge 2/5$, then the $\pi$-image of the 4-cycle $0110$ for $\sigma$ lies outside the hole. The next one to appear is $01101001$. etc. 

Once we have run through all these $2^n$-cycles, we reach another critical value, $a_*\approx 0.412454$ whose binary expansion is the famous Thue-Morse sequence:
\[
0110\ 1001 \ 1001 \ 0110 \ 1001 \ 0110 \dots
\]
It is known that if $a<a_*$, then $\mathcal J_{H_a}(T)$ is infinite countable; if $a>a_*$, then $\dim_H(\mathcal J_{H_a}(T))>0$ -- see \cite{ZB, Nilsson, GS}\footnote{Technically, \cite{GS} deals with the space of unique $\be$-expansions but the symbolic model is essentially the same.}. This calls for another definition:

\begin{Def}\label{def:second}
We say that a hole $H_0$ is {\em second order critical} for $S$ if
\begin{enumerate}
\item for any hole $H$ such that $\ov{H_0}\subset H$ we have $\dim_H(\mathcal J_H(S))=0$;
\item for any hole $H$ such that $\ov H\subset H_0$ we have $\dim_H(\mathcal J_H(S))>0$.
\end{enumerate}
\end{Def}

\begin{rmk}Definitions~\ref{def:first} and \ref{def:second} are in the same spirit as those of functions $\varphi$ and $\chi$ in \cite{LaMo}.
\end{rmk}

Consequently, $(a_*,1-a_*)$ is a second order critical hole for $T$. We see that there is a substantial distance between the first order and second order symmetrical critical holes for the doubling map. If we regard $a$ as time, then one may say that it takes long for $T$ to accommodate all these cycles, quite in line with the standard notion of ``route to chaos via period doubling'' common in one-dimensional dynamics. For more details see \cite{ACS}.

Let us give the central definition of this note which in a way combines the previous ones:

\begin{Def}\label{def:super}
We say that a hole $H_0$ is {\em supercritical} for $S$ if
\begin{enumerate}
\item for any hole $H$ such that $\ov{H_0}\subset H$ we have $\mathcal J_H(S)\subset \mbox{Fix}(S)$;
\item for any hole $H$ such that $\ov H\subset H_0$ we have $\dim_H(\mathcal J_H(S))>0$.
\end{enumerate}
\end{Def}

In other words, a supercritical hole is a hole which is both of first and second order. %Note that it is easy to construct a supercritical hole of full measure: take, for instance, an algebraic  automorphism $\tau$ of the $m$-torus $\BT^m$, a closed $\tau$-invariant set $\La$ of Hausdorff dimension strictly between 0 and $m$ and define $H_0=\BT^m\setminus\La$. Then $H_0$ is clearly supercritical. Finding ``small'' supercritical holes for a given map is more interesting; later we will see that for the doubling map a supercritical hole must have length at least $1/4$ (Theorem~\ref{thm:complete}) -- see also Remark~\ref{rmk:beta} at the end of the note.
Since from here on we only consider the doubling map, we would like to simplify our notation and write $\mathcal J(H)$ instead of $\mathcal J_H(T)$ or simply $\mathcal J(a,b)$ if $H=(a,b)$. Our first result yields a simple family of supercritical holes:

\begin{prop}\label{prop:degenerate}
For each $a\in[0,1/4]$, the holes $H_0=(a,1/2)$ and $H_0'=(1/2, 1-a)$ are supercritical for the doubling map $T$.
\end{prop}
\begin{proof} We will prove the first claim; the second one is obtained by swapping 0 and 1. Let us check both conditions of Definition~\ref{def:super}.

\smallskip
\noindent
(i) Suppose $\ov{H_0}\subset H$; then $[1/4,1/2]\subset H$, which means that each point $x\in H$ whose $T$-orbit avoids $H$ cannot have $01$ in its binary expansion. Hence $x=0$, because if the binary expansion of $x$ ends with $10^\infty$, it can be replaced with $01^\infty$, which lies in $H$.

\medskip\noindent (ii) Let $\ov H\subset H_0$. Fix $n\ge2$ and consider the following subshift of finite type:
\[
\Sigma_n=\{w\in\Sigma : w_k=0\implies w_{k+j}=1, \ j=1,\dots, n\}.
\]
When each 0 in the binary expansion of some $x\in(0,1)$ is succeeded by at least $n$ consecutive 1s, this means that $x\ge \frac{2^{-2}+2^{-3}+\dots+2^{-n}}{1-2^{-n-1}}$. Thus,
\[
\pi(\Sigma_n)\subset \left[\frac{\frac12-2^{-n-1}}{1-2^{-n-1}}, 1\right),
\]
which implies that $\pi(\Sigma_n)\subset \mathcal J(H)$ for all $n$ large enough. The topological entropy $h_{top}$ of the subshift $\sigma|_{\Si_n}$ is positive, whence $\dim_H \mathcal J(H)\ge \dim_H\pi(\Sigma_n)=h_{top}(\sigma|_{\Si_n})/\log2>0$.
\end{proof}

In the next section we will construct a continuum of supercritical holes for $T$ parametrized by Sturmian sequences. In Section~\ref{sec:complete} we present a complete list of supercritical holes for $T$.

\section{Sturmian holes}

For our purposes we need to define Sturmian systems. We will adapt our exposition to our needs; for a detailed survey and proofs of the auxiliary results stated below see \cite[Chapter~2]{Loth}.

We say that a finite  word $u$ is a {\em factor of} $w$ if there exists $k$ such that $u=w_k\dots w_{k+n}$ for some $n\ge0$. For a finite word $w$ let $|w|$ denote the length of $w$ and let $|w|_1$ stand for the number of 1s in $w$.  A finite or infinite word $w$ is {\em balanced} if for any $n\ge1$ and any two factors $u,v$ of $w$ of length~$n$ we have $||u|_1-|v|_1|\le1$.

Let $\ga\in(0,1/2)\setminus\mathbb Q$ and let its continued fraction expansion be denoted by $[d_1+1,d_2,d_3,\dots]$ with $d_1\ge1$ (in view of $\ga<1/2$). Let $p_n/q_n$ stand for the finite continued fraction $[d_1+1,\dots,d_n]$ (in least terms). We define the sequence of 0-1 words given by $\ga$ as follows: $s_{-1}=1, s_0=0,
s_{n+1}=s_n^{d_{n+1}}s_{n-1}, \ n\ge0$. The word $s_n$ is called the $n$th {\em standard word} given by $\ga$.

It is well known that $|s_n|_1=p_n$ and $|s_n|=q_n$ for every $n \geq 1$. Since $s_{n+1}$ begins with $s_n$ and $q_n\to\infty$, we conclude that there exists the limit $s_\infty(\ga)=\lim_{n\to\infty} s_n$. This word is called {\em the characteristic word given by $\ga$}. Note that any factor of a characteristic word is balanced.

Define the {\em Sturmian system given by $\ga$} as follows:
\[
X_\ga=\overline{\{\sigma^n s_\infty(\ga) : n\in\BN\}}.
\]
It is known that $X_\ga$ is a perfect set and $\sigma|_{X_\ga}$ is minimal and has zero entropy. It is also known that any sequence $w$ in $X_\ga$ which starts with 1 is a combination of blocks $10^{d_1}$ and $10^{d_1+1}$. If $w$ starts with 0, it actually starts with $0^{d_1}1$ or $0^{d_1+1}1$, with the above rule valid for the rest of $w$. In particular, the blocks of 0's are bounded for a given $\ga$.

It is obvious that $\pi|_{X_\ga}$ is an injection. From the above it follows that $K_\ga:=\pi(X_\ga)$ is a $T$-invariant Cantor set of zero Hausdorff dimension in $(0,1)$ which does not contain $1/2$ (since $1/2$ has the binary expansion with an unbounded number of 0's or 1's).

\begin{rmk}The set $K_\ga$ has been considered in \cite{BS} where the authors proved in particular that any ordered $T$-invariant subset of $[0,1]$ must be $K_\ga$ for some $\ga$ or its finite version for $\ga\in\mathbb Q$.
\end{rmk}

Recall that for $w,w'\in\Sigma$ we have that $w$ is {\em lexicographically less} than $w'$ if $w_k<w_k'$ for the smallest $k$ such that $w_k\neq w_k'$. (Notation: $w\prec w'$.) It is clear that $w\preceq w'\iff \pi(w)\le\pi(w')$.

Let $H_0(\ga)$ denote the hole in $K_\ga$ which contains $1/2$ (it is also the largest hole in $K_\ga$). More precisely, let 0-$\max(\ga)$ denote the lexicographically largest sequence in $X_\ga$ which starts with 0; similarly, we define 1-$\min(\ga)$. It is known (see \cite[Chapter~2]{Loth}) that the minimum sequence in $X_\ga$ is in fact $0s_\infty(\ga)$ and the maximum is $1s_\infty$. Consequently, 0-$\max(\ga)=01s_\infty(\ga)$ and 1-$\min(\ga)=10s_\infty(\ga)$, and $H_0(\ga)=(\pi(01s_\infty(\ga)), \pi(10s_\infty(\ga)))$. Clearly, the length of $H_0(\ga)$ equals $1/4$.

\begin{Def}We call $H_0(\ga)$ the {\em Sturmian hole} associated with $\ga$.
\end{Def}

\begin{example}Let $\ga=(3-\sqrt5)/2=[2,1,1,1,\dots]$. Here $s_{n+1}=s_ns_{n-1}$, which implies that $s_\infty(\ga)$ is none other than the Fibonacci word
\[
f=010010100100101001010\dots
\]
Hence
\begin{align*}
\pi^{-1}H_0((3-\sqrt5)/2) = (&01010010100100101001010\dots, \\ &10010010100100101001010\dots),
\end{align*}
and $H_0\left(\frac{3-\sqrt5}2\right)\approx(0.322549, 0.572549)$.
\end{example}

\begin{Def}
We say that a closed interval $I$ is a {\em trap} if $\bigcup_{k\ge0}T^{-k}(I)=(0,1)$.
\end{Def}

\begin{lemma}\label{lem:trap}
Let $\ga=[d_1+1,d_2,\dots]$. Fix $n\ge2$ and put $r=[d_1+1,\dots, d_n]$. Set $s=s(r)=01s_n, t=t(r)=10s_n$. Then $[\pi(s^\infty), \pi(t^\infty)]$ is a trap.
\end{lemma}
\begin{proof}Let $r_1<r<r_2$ be the Farey parents for $r$ and let $s_i=s(r_i), t_i=t(r_i), i=1,2$. It is shown in \cite[Lemma~3.2]{HS} that 
\begin{align*}
s&=s_1t_2,\\
t&=t_2s_1.
\end{align*}
To prove the claim, we will use induction of the length of $s$. If $|s|=2$, then we have $r=\frac12$, whence $[\pi(s^\infty), \pi(t^\infty)]=\left[\frac13,\frac23\right]$, about which we know that it is a trap -- see above. 

Assume now that the claim holds for all the rationals $r'$ for which $|s(r')|<s(r)$. Consequently, it holds for $s_1$. Without loss of generality we may assume $x\in[\pi(s_1^\infty),\pi(t_1^\infty)]$. If $x\in[\pi(s^\infty),\pi(t^\infty)]$, we are done; otherwise note that $s_1^\infty\prec s^\infty$ and $t_1^\infty\prec t^\infty$, which is a well known property of standard words -- see \cite[Chapter~2]{Loth}. Hence $x$ must lie in $[\pi(s_1^\infty), \pi(s^\infty))$. Put $y=T^{|s_1|}x\in [\pi(s_1^\infty), \pi(t_1^\infty))$. Again, if $y\in[\pi(s^\infty), \pi(t^\infty)]$, we are done, otherwise apply $T^{|s_1|}$ to $y$, etc.

If this process never stops, it means that $x=\pi(s_1^\infty)$. Let us show that its $T$-orbit falls in $[\pi(s^\infty),\pi(t^\infty)]$. Indeed, since $s_1$ is a cyclic permutation of $t_1$, there exists $j$ such that $\sigma^j(s_1^\infty)=t_1^\infty\prec t^\infty$. Also, since $t_1$ begins with 1, we have $\pi(t_1^\infty)>\frac12$. Hence $T^j(\pi(s_1^\infty))\in(\pi(s^\infty), \pi(t^\infty))$.
\end{proof}

\begin{thm}\label{thm:doubling}
 A Sturmian hole $H_0(\ga)$ is supercritical for the doubling map $T$ for any irrational $\ga\in(0,1/2)$.
\end{thm}

\begin{proof} (i) We want to show first that
\begin{equation}\label{eq:Kgamma}
K_\ga = \mathcal J(H_0(\ga))\setminus\{0\}.
\end{equation}
Note that
\[
\mathcal J(H_0(\ga))\setminus\{0\}=(0,1)\setminus \bigcup_{n=0}^\infty T^{-n} H_0(\ga).
\]
Firstly, we show that
\begin{equation}\label{eq:disjoint}
K_\ga \cap \bigcup_{n=0}^\infty T^{-n} H_0(\ga)=\varnothing.
\end{equation}
Since $K_\ga$ is $T$-invariant, we have $T^{-n}K_\ga\supset K_\ga$, whence
it suffices to show $T^{-n}K_\ga\cap T^{-n} H_0(\ga)=\varnothing$ for all $n\ge0$. This follows from $H_0(\ga)\cap K_\ga=\varnothing$, which proves (\ref{eq:disjoint}).

Let within this proof $s_n(\ga)$ denote the standard word corresponding to $[d_1+1,\dots, d_n]$. (Thus, $s_n(\ga)\to s_\infty(\ga)$ as $n\to\infty$.) Since $H_0(\ga)$ is the limit of $(\pi(01(s_n(\ga))^\infty), \pi(10(s_n(\ga))^\infty))$, and the these intervals approach $H_0(\ga)$ from both ends, it follows from Lemma~\ref{lem:trap} that $\overline{H_0(\ga)}$ is a trap. Hence $(0,1)\setminus \bigcup_{k\ge0}T^{-k}(H_0(\ga))$ is a closure of the orbit of its endpoints, i.e., $K_\ga$. This proves (\ref{eq:Kgamma}).

Now suppose $\ov{H_0(\ga)}\subset H$; we have $\mathcal J(H_0(\ga))\supset \mathcal J(H)$, whence by (\ref{eq:Kgamma}), $\mathcal J(H)\setminus\{0\}\subset K_\ga$. If $x\neq0$ and $x\in \mathcal J(H)$, then $x\in K_\ga$, and by the minimality of $T|_{K_\ga}$ and the fact that $\min H_0(\ga)\in K_\ga$, there exists $n\in\BN$ such that $T^n(x)\in (\min H, \min H_0(\ga))$, which contradicts $x\in \mathcal J(H)$. Hence $\mathcal J(H)=\{0\}$.

\medskip\noindent (ii) We define $\widetilde{X}_{\ga,n}$ to be the set of all sequences constructed out of the blocks $s_n:=s_n(\ga)$ and $s_{n+1}:=s_{n+1}(\ga)$ and put $X_{\ga,n}=\ov{\{\sigma^k w\mid k\ge0, w\in \widetilde{X}_{\ga,n}\}}$. Thus, each sequence in $X_{\ga,n}$ is of the form $w'w$, where $w'=\sigma^ks_n$ or $\sigma^k s_{n+1}$ and $w$ is a block sequence whose each block is either $s_n$ or $s_{n+1}$.

We claim that
\begin{equation}\label{eq:limit}
X_{\ga,n}\to X_\ga, \quad n\to\infty, \quad \text{in the Hausdorff metric}.
\end{equation}
Since $\sigma|_{X_\ga}$ is minimal, for each sequence $u\in X_\ga$ and each $\de>0$ there exists $k\ge0$ such that $\text{dist}(u,\sigma^k s_\infty(\ga))<\de$, where, as usual, $\text{dist}(x,y)=2^{-\min\{j\ge1\,\mid\, x_j\neq y_j\}}$. To prove (\ref{eq:limit}), we first show that for each $k\ge0$ there exists $w\in X_{\ga,n}$ such that
\begin{equation}\label{eq:approx}
\text{dist}(w,\sigma^k s_\infty(\ga))<\de.
\end{equation}
Suppose $k<q_n$. Then the sequence $w=\sigma^k(s_n)(s_n)^\infty$ satisfies (\ref{eq:approx}) with $\de=2^{-q_{n-1}}$. Indeed, $s_\infty(\ga)$ always begins with $s_ns_{n-1}\dots$: if $d_{n+1}=1$, then $s_{n+1}=s_ns_{n-1}$ and if $d_{n+1}\ge2$, then $s_{n+1}=s_ns_n*$, and $s_n$ always begins with $s_{n-1}$. Hence $w$ and $\sigma^k s_\infty(\ga)$ agree at least on the first $|s_{n-1}|=q_{n-1}$ coordinates.

On the other hand, if $w\in X_{\ga,n}$, then $w$ begins with either $(\sigma^k s_n) s_{n-1}$ with $k<q_n$ or $(\sigma^k s_{n+1}) s_n$ with $k<q_{n+1}$. Since both $s_ns_{n-1}$ and $s_{n+1}s_n$ are prefixes of $s_\infty(\ga)$, we conclude that in either case $\sigma^k(s_\infty(\ga))$ is at a distance less than or equal to $2^{-q_{n-1}}$ from $w$ again.

Now suppose $\ov H\subset H_0(\ga)$. By (\ref{eq:limit}), there exists $n$ such that $\pi(X_{\ga,n})\subset \mathcal J(H)$. Hence $\dim_H(\mathcal J(H))>0$.
\end{proof}

\begin{rmk}Note that as $H_0(\ga)\to (1/4, 1/2)$ from Proposition~\ref{prop:degenerate} as $\ga\to0$. This follows from the fact that $s_\infty(\ga)$ begins with $0^{d_1}$ and consequently, tends to the zero sequence as $\ga\downarrow0$.
Since a characteristic word can begin with $(01)^n$ for arbitrarily large $n$, we have $\sup_\gamma H_0(\ga)=1/3$. Therefore,
\begin{equation}\label{eq:K}
\mathcal K:=\left\{\min H_0(\ga) \mid \ga\in(0,1/2)\setminus \mathbb Q\right\}\subset (1/4,1/3).
 \end{equation}
Since the number of balanced words grows polynomially (see, e.g.,  \cite[Corollary~18]{Mig}), it is easy to see that $\mathcal K$ is a nowhere dense set of zero Hausdorff dimension.
\end{rmk}

\section{A complete description of supercritical holes for the doubling map}
\label{sec:complete}

Put
\[
\mathcal S=\{a\in[0,1) \mid \exists b>a \ \text{such that}\ (a,b)\ \text{is supercritical for}\ T\}.
\]

It follows from Proposition~\ref{prop:degenerate} that $[0,1/4]\subset \mathcal S$ and $1/2\in\mathcal S$. It is also clear that if $a>1/2$, then $a\notin\mathcal S$, since $\mathcal J(a,b)$ is of positive dimension for any $b>a$ -- see the proof of Proposition~\ref{prop:degenerate}~(ii).

Furthermore, since $(1/3, 2/3)$ is first-order critical for $T$, it is obvious that if $a\in\mathcal S$ and $a\ge1/3$, then the corresponding $b$ should be greater than $2/3$, i.e., $a+b>1$. Again, by symmetry, $(1-b,1-a)$ is supercritical, whence it suffices to confine our search for the elements of $\mathcal S$ to the interval $(1/4,1/3]$. %Finally, $1/3\notin\mathcal S$, because $(1/3, 2/3)$ is first-order critical, whereas $\mathcal J_{(a,b)}$ is countable for any $a\in(1/3, 2/5)$ and any $b\in(3/5, 2/3)$.

Our goal is to prove that
\begin{equation}\label{eq:S14}
\mathcal S\cap [1/4, 1/3]=\ov{\mathcal K},
\end{equation}
where $\mathcal K$ is given by (\ref{eq:K}).

Let $SW(a_1,\dots, a_n)$ denote the set of all standard words of length~$n$ given by some irrational $\ga$ whose continued fraction expansion begins with $[a_1+1,\dots, a_n]$. Thus, $s_{-1}=1, s_0=0, s_{k+1}=s_k^{a_{k+1}}s_{k-1}$ for all $k\in\{0,1,\dots, n-1\}$. Since for $a_n>1$, we have $[a_1,\dots, a_{n-1}, a_n]=[a_1,\dots, a_{n-1}, a_n-1, 1]$, we assume $a_n\ge2$.

Using this notation, we obtain
\[
(1/4, 1/3)\setminus \mathcal K = \bigcup_{p/q<1/2, (p,q)=1} [\alpha_{p/q}, \be_{p/q}],
\]
where $[\alpha_{p/q}, \be_{p/q}]$ is the gap between $\{\pi(01w) : w\in SW(a_1,\dots, a_{n-1}, a_n)\}$ and $\{\pi(01w): w\in SW(a_1,\dots, a_{n-1}, a_n-1)\}$.

\begin{example}\label{ex1}
Let $p/q=1/3$, which implies $n=1$ and $a_1=2$. Denote $s_1=001, s_1'=01$.

Then $s_2=(01)^{d_2}0$, and $s_3=((01)^{d_2}0)^{d_3}01$, which is always lexicographically larger than $(010)^\infty$ (corresponding to $d_2=1, d_3=\infty$).
Similarly, $s_2'=(001)^{d_2}0$, and the supremum corresponds to $d_2=\infty$, which yields $(001)^\infty$. Therefore,
\[
[\alpha_{1/3}, \be_{1/3}]=[\pi(01(001)^\infty),\pi(01(010)^\infty)]=\left[\frac27, \frac9{28}\right].
\]
Note that $01(001)^\infty=(010)^\infty$, i.e., purely periodic. Besides, $\alpha_{1/3}$ and $\be_{1/3}$ are points in the same periodic orbit. We will see later that this is always the case for $\alpha_{p/q}$.
\end{example}

\begin{example}\label{ex2}
Let now $n=2$ and $p/q=2/5=[2,2]$. Here $s_1=01, s_2=01010, s_2'=010$. We observe the role-reversal: it is now $[2,1]$ that yields $\alpha_{2/5}$ and $[2,2]$ yields $\be_{2/5}$. More precisely,
\begin{align*}
\sup\{s_N : a_1=1, a_2=1\}&=(01001)^\infty, \\
 \inf\{s_N : a_1=1, a_2=2\}&=(01010)^\infty,
\end{align*}
whence, in view of $01(01001)^\infty=(01010)^\infty$,
\[
[\alpha_{2/5}, \be_{2/5}]=[\pi((01010)^\infty),\pi(01(01010)^\infty)]= \left[\frac{10}{31},\frac{41}{124}\right].
\]
\end{example}

Before we tackle the general case, we need the following auxiliary claim:

\begin{lemma}\label{lem:sn1}
Let $s_{n-1}, s_n$ be standard words with $s_n=s_{n-1}^{a_n}s_{n-2}$. Then
\begin{align*}
s_{n-1}s_n&\prec s_ns_{n-1},\quad n\ \text{odd,}\\
s_ns_{n-1}&\prec s_{n-1}s_n,\quad n\ \text{even.}
\end{align*}
\end{lemma}
\begin{proof}We prove both claims simultaneously by induction. It is obvious that $s_0s_{-1}=01\prec10=s_{-1}s_0$ and $s_0s_1=0^{d_1+1}1\prec 0^{d_1}10=s_1s_0$. Assume the claim holds for all $k<n$. We have
\[
s_{n-1}s_n=s_{n-1}^{d_n+1}s_{n-2}= s_{n-1}^{d_n}s_{n-1}s_{n-2}.
\]
If $s_{n-1}s_{n-2}\prec (\text{resp.}\ \succ) s_{n-2}s_{n-1}$, this implies $s_{n-1}s_n\prec (\text{resp.}\ \succ) s_ns_{n-1}$.
\end{proof}

The following lemma gives an explicit formula for $[\alpha_{p/q}, \be_{p/q}]$ for a general $p/q$.

\begin{lemma}\label{lem:alphapq}
We have
\[
[\alpha_{p/q}, \be_{p/q}]=
\begin{cases}
[\pi(01 (s_n)^\infty), \pi(01(s_{n-1}^{a_n-1}s_{n-2}s_{n-1})^\infty)],& n\ \text{ odd},\\
[\pi(01(s_{n-1}^{a_n-1}s_{n-2}s_{n-1})^\infty), \pi(01 (s_n)^\infty)] ,& n\ \text{ even}.
\end{cases}
\]
\end{lemma}
\begin{proof}We show first that
\begin{equation}\label{eq:sup}
\sup\{SW(d_1,d_2,\dots) \mid d_1=a_1,\dots d_n=a_n\}=(s_n)^\infty,\quad n\ \text{odd},
\end{equation}
where supremum is understood in the sense of lexicographic order. We have $s_{n+1}=s_n^{d_{n+1}}s_{n-1}$, and $s_{n+2}= s_{n+1}^{d_{n+2}}s_n=s_n^{d_{n+1}}s_{n-1}s_n^{d_{n+1}}\ldots \prec (s_n)^\infty$, in view of Lemma~\ref{lem:sn1}. (Since $s_n$ begins with $s_{n-1}$, we have $s_{n-1}s_n\dots\prec s_ns_n\dots$). This proves (\ref{eq:sup}).

Similarly,
\begin{equation}\label{eq:sup2}
\sup\{SW(d_1,d_2,\dots) \mid d_1=a_1,\dots d_n=a_n\}=(s_ns_{n-1})^\infty,\quad n\ \text{even}.
\end{equation}
Indeed, by Lemma~\ref{lem:sn1}, we have $s_n^2\prec s_{n-1}s_n$. Hence $s_n^{d_{n+1}}s_{n-1}s_n^{d_{n+1}}\ldots \prec (s_ns_{n-1})^\infty$, which proves (\ref{eq:sup2}). Consequently,
\begin{align*}
\sup&\{SW(d_1,d_2,\dots) \mid d_1=a_1,\dots, d_{n-1}=a_{n-1}, d_n=a_n-1\}\\
&=(s_{n-1}^{a_n-1}s_{n-2}s_{n-1})^\infty,\ \ n\ \text{even}.
\end{align*}
The two remaining inequalities are proved in exactly the same way.
\end{proof}

By induction, $s_n$ always ends with $01$ if $n\ge1$ is odd and with $10$ otherwise. Hence $\alpha_{p/q}$ always has a purely periodic binary expansion of period $q$. Put $\ga_{p/q}=\be_{p/q}+1/4$. Clearly, the binary expansion of $\ga_{p/q}$ can be obtained from that of $\be_{p/q}$ by replacing the first two digits 0,1 with 1,0. Hence by the previous lemma, the binary expansion of $\ga_{p/q}$ is purely periodic of period $q$ as well. Thus, we have

\[
T^{q}(\alpha_{p/q})=\alpha_{p/q},\quad
T^{q}(\beta_{p/q})=T^{q}(\ga_{p/q})=\ga_{p/q}.
\]

Put
\[
\mathcal O(p/q)=\{T^n (\alpha_{p/q}) \mid n\in\BN\}.
\]

\begin{lemma}\label{lem:minmax}
We have $\mathcal O(p/q)\cap (\alpha_{p/q}, \ga_{p/q})=\varnothing$.
\end{lemma}
\begin{proof}Since $\ga_{p/q}$ is purely periodic, and the binary expansion of the periodic part has the length $q$ and the same number of 1s as that of $\alpha_{p/q}$ (i.e., $p$), we conclude that $\ga_{p/q}\in\mathcal O(p/q)$, in view of the fact that any two balanced words with the same length and the same number of 1s are cyclic permutations of each other -- see \cite[Chapter~2]{Loth}.

Assume first $n$ to be odd; we have
\[
\alpha_{p/q}=\pi(01(s_n)^\infty)=\pi((w_1...w_{q})^\infty). \]
It is well known that if $w=w_1\dots w_q$ is standard and ends with 01, then its cyclic permutation $w_qw_1\dots w_{q-1}$ is lexicographically largest among all cyclic permutations of $w$. Consequently, $w_{q-1}w_qw_1\dots w_{q-2}$ is the largest cyclic permutation beginning with 0.

The word $s_{n-1}^{a_n-1}s_{n-2}s_{n-1}$ has the same length and the same number of 1s as $s_n$ (and is balanced, since it is standard), whence it is also a cyclic permutation of $s_n$. By the same argument, adding the prefix 10 to $w_1'\dots w_q':=s_{n-1}^{a_n-1}s_{n-2}s_{n-1}$ and removing its last two symbols yields the lexicographically smallest cyclic permutation $10w'_1\dots w'_{q-2}$ of $s_n$ beginning with 1.

Thus, there is no cyclic permutation which is lexicographically greater than $01w_1\dots w_{q-2}$ and less than $10w'_1\dots w'_{q-2}$. This implies that there is no element of $\mathcal O(p/q)$ between $\alpha_{p/q}$ and $\ga_{p/q}$.

The case of even $n$ is analogous.
\end{proof}

\begin{cor}\label{cor:trap}
For each $p/q$ we have
\[
\mathcal J(\alpha_{p/q}, \ga_{p/q})=\mathcal O(p/q) \cup \bigcup_{k\ge0} T^{-k}(\{\alpha_{p/q}, \ga_{p/q}\}).
\]
\end{cor}
\begin{proof}The claim follows from Lemma~\ref{lem:minmax} together with Lemma~\ref{lem:trap}.
\end{proof}

\begin{prop}For any $p/q$ we have
\[
(\alpha_{p/q}, \be_{p/q})\cap \mathcal S=\varnothing.
\]
\end{prop}

\begin{proof}Let $a\in (\alpha_{p/q}, \be_{p/q})$ and assume, on the contrary, that there exists $b$ such that $(a,b)$ is supercritical. Note first that $b$ cannot be less than $\ga_{p/q}$, otherwise Lemma~\ref{lem:minmax} implies that there exists $\de>0$ such that $\mathcal O(p/q)\subset \mathcal J(a-\de,b+\de)$. Thus, $b\ge \ga_{p/q}$.

We claim that the condition of Definition~\ref{def:super}~(ii) fails here. Namely, since $T^{q}(\alpha_{p/q})=\alpha_{p/q}, \ T^{q}(\be_{p/q})=\ga_{p/q}$, we have that there exists a unique  $k\ge1$ such that $a':=T^{kq}(a)\in (\be_{p/q},\ga_{p/q})$. Let $\ga'(p/q)$ be such that $T^{q}(\ga'_{p/q})=\be_{p/q}$.

\begin{figure}[t]
\centering \unitlength=1.3mm
\begin{picture}(40,70)(0,0)

\thinlines

\path(-10,5)(50,5)(50,65)(-10,65)(-10,5)

\dottedline(-10,5)(50,65)
\dottedline(10,5)(10,65)

\put(-12,2){$\alpha_{p/q}$}
\put(9,2){$\be_{p/q}$}
\put(48,2){$\ga_{p/q}$}

\thicklines

\path(-10,5)(10,65)
\path(30,5)(50,65)

\thinlines

\dottedline(10,25)(36.5,25)
\dottedline(36.5,25)(36.5,5)

\put(34,2){$\ga'_{p/q}$}

\dottedline(8,5)(8,58)
\dottedline(8,58)(43,58)
\dottedline(43,58)(43,5)

\put(41.6,2){$a'$}

\put(7,2){$a$}

\put(22,2){$\alpha_{p/q}+\frac14$}

\end{picture}

\caption{Part of the map $T^{q}$ restricted to the interval $[\alpha_{p/q}, \ga_{p/q}]$}
    \label{fig0}
  \end{figure}
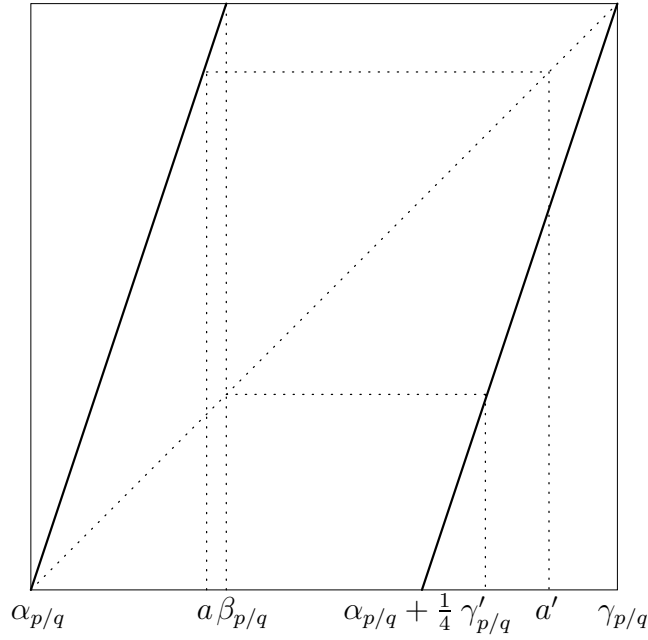

Then

\begin{equation}\label{eq:trap}
\mathcal J(a,b)=\mathcal O(p/q)\quad \forall b\in (\max\{a',\ga'(p/q)\},\ga_{p/q}].
\end{equation}
 Indeed, by Corollary~\ref{cor:trap}, it suffices to check  that for any $x\in (\alpha_{p/q},a)\cup (\max\{a',\ga'(p/q)\},\ga_{p/q})$ there exists $s\in\BN$ such that $T^s(x)\in (a,b)$. If $x\in (\alpha_{p/q},a)$, then we put $s=kq$; if $x\in (\max\{a',\ga'(p/q)\},\ga_{p/q})$, then we put $s=q$ -- see Figure~\ref{fig0}.\footnote{There may be pieces of the map $T^{q}$ between $\be_{p/q}$ and $\alpha_{p/q}+1/4$ which lie in the square in question but they do not affect our argument, so we ignore them.} This proves (\ref{eq:trap}).

Thus, our assumption was wrong, which proves the proposition.
\end{proof}

The only points to study which remain are the $\alpha_{p/q}, \be_{p/q}$ and $1/3$. We can use the following simple claim:

\begin{lemma}Suppose $a_j\to a,\ b_j\to b$ as $j\to+\infty$, and that each hole $(a_j, b_j)$ is supercritical for $T$. Then $(a,b)$ is supercritical as well.
\end{lemma}
\begin{proof}Follows directly from Definition~\ref{def:super}.
\end{proof}

Now, we know that both $\alpha_{p/q}$ and $\be_{p/q}$ are limit points of the set $\mathcal K$, which implies $(\alpha_{p/q}, \alpha_{p/q}+1/4)$ and $(\be_{p/q}, \be_{p/q}+1/4)$ are supercritical. Consequently,  $\alpha_{p/q}\in\mathcal S$ and $\be_{p/q}\in\mathcal S$ for all $p/q$.

\begin{example}It follows from Examples~\ref{ex1} and \ref{ex2} that the holes $\left(\frac27, \frac{15}{28}\right), \left(\frac9{28}, \frac47\right), \left(\frac{10}{31},\frac{71}{124}\right)$ and $\left(\frac{41}{124},\frac{18}{31}\right)$ are all supercritical.
\end{example}

Finally, $1/3=\pi((01)^\infty)$, which is $\lim_{N\to\infty}\pi(01 s_\infty(\ga_N))$, where $\ga_N=[2,N,1,1,1,\dots]$.

Therefore,
\[
\mathcal S\cap [1/4, 1/3]=\{1/4\}\cup\mathcal K \cup\bigcup_{p/q<1/2, (p,q)=1}\{\alpha_{p/q}, \be_{p/q}\}\cup\{1/3\}.
\]
Note that
\[
\ov{\mathcal K}=\{1/4\}\cup \mathcal K \cup\bigcup_{p/q<1/2, (p,q)=1}\{\alpha_{p/q}, \be_{p/q}\},
\]
which yields (\ref{eq:S14}).

It follows from Proposition~\ref{prop:degenerate} that $(1/2,b)$ is supercritical for $T$ for any $b\in[3/4,1]$. The following proposition asserts that if $a\in\mathcal S\setminus\{1/2\}$, then the right endpoint of the corresponding supercritical hole cannot ``float'':

\begin{prop}\label{prop:setS}
If $a\in\mathcal S\setminus\{[0,1/4]\cup\{1/2\}\}$ and $(a,b)$ is supercritical for $T$, then $b=a+1/4$. If $a\in[0,1/4]$ and $(a,b)$ is supercritical for $T$, then $b=1/2$.
\end{prop}
\begin{proof} We have several cases. Within this proof we always assume that $(a,b)$ is supercritical for $T$.

\smallskip\noindent
(a) Let first $a\in[0,1/4]$. If $b<1/2$, then we know $\mathcal J(a+\de, b-\de)$ has a positive dimension for all $\de<1/2-b$, which means $(a,b)$ cannot be supercritical.

Assume $b>1/2$. If $a<1/4$, then $\mathcal J(a+\de, b-\de)=\{0\}$ for all $\de<\min\{a-1/4, 1/2-b\}$. If $a=1/4$, then $T(1/4+\de)=1/2+2\de\in (1/4+\de, b-\de)$ for any $\de<(b-1/2)/3$.

\medskip\noindent
(b) Now let $a=\alpha_{p/q}$ for some $p/q$ and assume first that $b>a+1/4$. Consider the hole $(a+\de, b-\de)$ with $b-\de>a+1/4$. Since $(a,a+1/4)$ is supercritical, $\mathcal J(a-\e,a+1/4+\e)=\{0\}$ for any $\e>0$, whence $\mathcal J(a+\de, b-\de)=\{0\}$, because if $\de$ is small enough, then for any $x\in(a,a+\de)$ there exists $j\ge1$ such that $T^j x\in (a+\de, b-\de)$. This contradicts $(a,b)$ being supercritical.

Now assume $b<a+1/4$. Since $(a,a+1/4)$ is supercritical, $\mathcal J(a+\de, a+1/4-\de)$ is of positive dimension, therefore, non-empty. However, no $x\in(a,a+\de)$ is in $\mathcal J(a+\de, a+1/4-\de)$, which means that there exists $y\in(b, a+1/4)$ such that the $T$-orbit of $y$ has an empty intersection with $(a+\de, a+1/4-\de)$ and consequently, with $(a+\de, b-\de)$. Hence $(a,b)$ cannot be supercritical. The case $a=\be_{p/q}$ is similar.

\medskip\noindent
(c) If $a=1/3$ and $b\neq7/12$, then the argument goes exactly like case~(b), and we leave it to the reader as an exercise.

\medskip\noindent
(d) Finally, let $a\in\mathcal K$ be given by $\ga=[a_1+1,a_2,\dots]$, and, as usual, $s_n=s_{n-1}^{a_n}s_{n-2}, n\ge1$. Note first that $(s_n)^\infty\prec s_\infty(\ga)$ if $n$ is even and $(s_n)^\infty\succ s_\infty(\ga)$ otherwise. Indeed, $s_\infty(\ga)$ begins with $s_{n+1}s_n=s_n^{d^{n+1}}s_{n-1}s_n$, and $s_n$ begins with $s_{n-1}$, so the inequalities in question follow from Lemma~\ref{lem:sn1}. Hence by Lemma~\ref{lem:alphapq}, $a_n:=\alpha_{p_{2n-1}/q_{2n-1}}\downarrow a$ and $a_n':=\alpha_{p_{2n}/q_{2n}}\uparrow a$ as $n\to\infty$.

Assume now $b>a+1/4$; then there exists $n$ such that $a_n+1/4<b$, while $(a_n,a_n+1/4)$ is, as we know, supercritical. Since $[a_n,a_n+1/4]\subset (a,b)$, this contradicts Definition~\ref{def:super}. Similarly, if $b<a+1/4$, then $[a,b]\subset (a_n', a_n'+1/4)$ for some $n$, which is equally impossible if $(a,b)$ is supercritical.
\end{proof}

\begin{rmk}The case $a=1/3$ is rather curious. We know that the 2-cycle arises at $b_0=2/3$ and that there are no other cycles in $\mathcal J(a,b)$ for $b\in(b_1,b_0)$, where $b_1=7/12$ (see \cite{ACS}). Once $b<b_1$, we immediately get from Proposition~\ref{prop:setS} that $\dim_H \mathcal J(a,b)>0$, which implies that $\mathcal J(a,b)$ contains infinitely many cycles. As we see, this is different from the symmetrical case.
\end{rmk}

As a corollary, we obtain a complete description of the set of all possible left endpoints of supercritical points:

\begin{cor}We have
\[
\mathcal S=[0,1/4)\cup\ov{\mathcal K}\cup(3/4-\ov{\mathcal K}),
\]
i.e., $\mathcal S$ is a disjoint union of an interval and a Cantor set of zero Hausdorff dimension.
\end{cor}
\begin{proof}It follows from Proposition~\ref{prop:setS} that if $(a,b)$ is supercritical for $T$, then no $(a',b)$ is such for $a'\neq a$ unless $b=1/2$.
\end{proof}

So, we have obtained a complete list of all supercritical holes for the doubling map:

\begin{thm}\label{thm:complete}
Each supercritical hole for the doubling map is one of the following:
\begin{enumerate}
\item $(\al, 1/2)$ or $(1/2, 1-\al)$ for any $\al\in[0,1/4]$;
\item $(\al,\al+1/4)$ or $(3/4-\al, 1-\al)$, where $\al$ is one of the following:
    \begin{itemize}
    \item $1/3$;
    \item $\pi(01 s_\infty(\ga))$, where $s_\infty(\ga)$ is the characteristic word for an arbitrary irrational $\ga<1/2$;
        \item $\pi(01 (s_n)^\infty)$ or $\pi(01(s_{n-1}^{a_n-1}s_{n-2}s_{n-1})^\infty)$, where $(s_k)_{k=-1}^n$ is the sequence of standard words parametrized by an arbitrary $n$-tuple $(a_1,\dots, a_n)\in\BN^n$ with $a_n\ge2$. The binary expansion of $\al$ is an eventually periodic sequence of period~$q$, where $p/q=[a_1+1,\dots, a_n]$.
        \end{itemize}
    \end{enumerate}
\end{thm}

\begin{rmk}\label{rmk:beta}
It would be interesting to find a complete list of supercritical holes for an arbitrary $\beta$-transformation $T_\be(x)=\be x\bmod1$ with $\be>1$. Note that some of the results of our paper can be easily transferred to such a map if $\be<2$; for instance, similarly to Proposition~\ref{prop:degenerate}, one can easily show that $H_\be:=(\be^{-2},\be^{-1})$ is supercritical for $T_\be$. As a corollary, we infer that a supercritical hole can have a measure arbitrarily close to 0, since $|H_\be|\to0$ as $\be\downarrow1$. However, the fact that not all 0-1 sequences are admissible (see, e.g., \cite{Pa}) makes this set-up significantly more complicated as far as a complete description is concerned.
 \end{rmk}

\medskip\noindent\textbf{Acknowledgement.} The author is grateful to Henk Bruin, Paul Glendinning, Kevin Hare and Oliver Jenkinson for stimulating discussions.

\end{document}